\documentclass[DIV=classic,a4paper,10pt]{myart}

\KOMAoptions{DIV=last}
\usepackage{atbegshi}
\AtBeginDocument{\AtBeginShipoutNext{\AtBeginShipoutDiscard}}

\usepackage{caption}
\usepackage{subcaption}

\newcommand{\diffns}{\mathrm{d}}
\usepackage{accents}
\newcommand{\ubar}[1]{\underaccent{\bar}{#1}}

\begin{document}

\title{Pathwise uniqueness of non-uniformly elliptic SDEs with rough coefficients}\thanks{We thank an anonymous referee for helpful suggestions.}
\tnotetext[t]{}

\author[]{Olivier Menoukeu-Pamen}
\author[]{Youssef Ouknine}
\author[]{Ludovic Tangpi}

\abstract{
In this paper we review and improve pathwise uniqueness results for some types of one-dimensional stochastic differential equations (SDE) involving the local time of the unknown process.
The diffusion coefficient of the SDEs we consider is allowed to vanish on a set of positive measure and is not assumed to be smooth.
As opposed to various existing results, our arguments are mainly based on the comparison theorem for local time and the occupation time formula.
We apply our pathwise uniqueness results to derive strong existence and other properties of solutions for SDEs with rough coefficients.
}

\date{\today}
\keyAMSClassification{60H10, 60H60, 60J55}
\keyWords{Stochastic differential equations, pathwise uniqueness, comparison theorem for local times, local time of the unknown.}
\maketitleludo
\setcounter{page}{1} %

\section{Introduction}
Let $T\in (0,\infty)$ be a fixed deterministic time horizon and $(\Omega, {\cal F}, P)$ a given probability space equipped with the completed filtration $({\cal F}_t)_{t\in [0,T]}$ of a $d$-dimensional Brownian motion $W$.
We denote by $L^a(X)$ the local time at level $a\in \mathbb{R}$ of the semimartingale $X$.
Given a signed Radon measure $\nu$ on (the Borel subsets of) $\mathbb{R}$ and a progressively measurable function $\sigma:[0,T]\times\Omega\times\mathbb{R}\to \mathbb{R}^d$, we are interested in studying pathwise uniqueness for solutions of the one-dimensional stochastic differential equation
\begin{equation}\label{eqmainlt}
	X_t = x + \int_0^t\sigma_u(X_u)\,\diffns W_u + \int_{\mathbb{R}}L_t^a(X)\,\nu(\diffns  a).
\end{equation}
Such equations first appeared in the work of \citet{LeGall83} and was subsequently developed e.g. by \citet{Engl-Schm123} and \citet{Sto-Ypr}. 
One strong interest in this type of equations involving the local time of the unknown is due to its link to the so-called skew Brownian motion introduced and studied by \citet{Harr-She} and \citet{Blei-Engl}.
Pathwise uniqueness results for the SDE \eqref{eqmainlt} was obtained by \citet{Ouknine88} when $\sigma$ is of bounded variation.
In the case when $\nu$ is $\sigma$-finite and the diffusion coefficient is time-homogeneous, \citet{Blei-Engl} derived necessary and sufficient conditions for existence and uniqueness in law of a solution.
More recently, \citet{Ben-Bou-Ouk13} derived pathwise uniqueness results using the balayage formula.
Recall that the relevance of pathwise uniqueness of SDEs is stressed by the celebrated result of \citet{Yam-Wata} which allows, from pathwise uniqueness and weak existence, to derive strong existence.

If the measure $\nu$ is absolutely continuous w.r.t. the Lebesgue measure, a direct application of the occupation time formula shows that the SDE \eqref{eqmainlt} can be rewritten as
\begin{equation}
\label{eq:main_no.lt}
	X_t = x + \int_0^t\sigma_u(X_u)\diffns W_u + \int_0^tf(X_u)\sigma_u^2(X_u)\diffns  u,
\end{equation}
where the measurable function $f:\mathbb{R}\to \mathbb{R}$ is the density of $\nu$.
In the one dimensional case, when the drift is bounded and Borel measurable, \citet{Zvonkin} derives existence and uniqueness of strong solution. This result was generalised to multidimensional case by \citet{Veret81}. Since then, there has been a strong research effort to derive \emph{existence, uniqueness} and \emph{regularity} properties for \emph{strong solutions} of SDEs with non-smooth coefficients; see for example \citet{Engl-Schm123,Eng-Sch85}, \citet{Kry-Roeck05}, \citet{Menoukeu-13} and the references therein.
A prevalent assumption in the literature to derive pathwise uniqueness (and strong existence) is a uniform ellipticity condition on the diffusion coefficient $\sigma$, that is, $\frac{1}{2}tr(\sigma_t\sigma_t')(x)\ge c$ for some $c>0$ for all $(t,x)$.

The main objective of the present work is to study properties of solutions of the SDE \eqref{eqmainlt} without any a priory assumption on uniform ellipticity of the diffusion coefficient which in turn will be assumed to be merely measurable.
In this setting, the question of pathwise uniqueness of \eqref{eqmainlt} was studied by \citet{Engl-Schm123} and more recently, \citet{Cham-Jabin17} derived strong existence and pathwise uniqueness results for classical SDEs (i.e. without local time) when $\sigma$ is allowed to vanish.
Making ample use of the theory of local time, --more precisely of the comparison theorem for local times of \citet{Ouknine90} and \citet{Ben-Bou-Ouk13}-- and using more simple arguments, we improve existing results on pathwise uniqueness of \eqref{eqmainlt}, giving simplified arguments. In particular, we show that the so-called condition (LT) of \citet{Bar-Per84} guarantees pathwise uniqueness of SDEs with local time even in a more general setting than that of \eqref{eqmainlt}; see Theorem \ref{thm:lt_local_nonhom} and Proposition \ref{pro:lt}.
Assuming that the diffusion is deterministic and time-homogeneous i.e., $\sigma_t(\omega,x)=\sigma(x)$,  we derive the well-known uniqueness result of \citet{Engl-Schm123} using the comparison theorem for local times and the occupation time formula, see Theorem \ref{thm:pathwise.unique}.
For illustration purpose, we show how comparison can also be used to derive uniqueness for reflected SDEs.

Using the result of \citet{Yam-Wata} along with a transformation that eliminates the drift, we derive, as applications of our uniqueness results, existence of strong solutions of \eqref{eqmainlt}.
We further study some properties of the solution, including continuity and PDE-representation.
In particular, we show the rather striking fact that even if the coefficients of the SDEs are not smooth, the solution is still a continuous function of time and of the initial condition.
This is in line with the results of \citet{Mohammed15}, \citet{Menoukeu-17} and \citet{Fed-Flan} obtained under uniform ellipticity and of \citet{Bah-Mer-Ouk98} in the case for classical SDEs.

The rest of this work is organized as follows:
In the next section, we study pathwise uniqueness of SDEs with local time of the unknown, first considering the time-inhomogeneous case and then the time-homogeneous case, and in the final section we apply our pathwise uniqueness results to derive properties of SDEs with local time as well as classical SDEs not involving the local time of the unknown.

\section{Pathwise uniqueness}
\label{sec:uniqueness}

\subsection{The time-inhomogeneous case}
In this section, we study SDEs of the form
\begin{equation}
\label{eq:SDE_local_nonhom}
	X_t = x + \int_0^t\sigma_u(X_u) \diffns  W_u + \int_0^t\int_{\mathbb{R}}\diffns L^a_u(X)\nu_u(\diffns a),
\end{equation}
where $\sigma:[0,T]\times \Omega\times \mathbb{R}\to \mathbb{R}^d$ is a progressively measurable function, $(\nu_t)_{t\in [0,T]}$ is a flow of Borel measures on $\mathbb{R}$, $W$ is a $d$-dimensional Brownian motion and $L^a(X)$ the local time at level $a\in \mathbb{R}$ of the process $X$.
The SDE \eqref{eq:SDE_local_nonhom} was studied in \cite{Weinryb} when $\sigma$ is constant and $\nu_t$ of the form $\alpha_t\delta_0(da)$ where $\delta_0$ is the Dirac mass at zero; see also Subsection \ref{subsec:exist}.
It is well known that uniqueness in law is weaker than pathwise uniqueness.
Our first result gives a condition under which the converse holds true for the SDE \eqref{eq:SDE_local_nonhom}.
The following condition was introduced in \cite{Bar-Per84} and further considered in \cite{Rut90}:
\begin{definition}[Condition (LT)]
	The function $\sigma$ satisfies the condition (LT) if $L^0_t(X^1 - X^2)=0$ for all $t \in [0,T]$ and for every processes $X^1, X^2$ such that
	\begin{equation}
	\label{eq:lt}
		X^i_t = X^i_0 + \int_0^t\sigma_u(X^i_u) \diffns  W_u + V^i_t, \quad i=1,2,
	\end{equation}
	where $(V^i_t)_{t \in [0,T]}$ are continuous adapted processes with bounded variation.
\end{definition}
\begin{theorem}
\label{thm:lt_local_nonhom}
Suppose $\sigma$ satisfies (LT), and the SDE \eqref{eq:SDE_local_nonhom} satisfies uniqueness in law. Then it satisfies pathwise uniqueness.
\end{theorem}
\begin{proof}
	Using the condition (LT), we can show that if $X^1$ and $X^2$ are solutions of \eqref{eq:SDE_local_nonhom} so are $X^1\wedge X^2$ and $X^1\vee X^2$.
	In fact, since $\int_0^t\int_{\mathbb{R}}\diffns L^a_u(X)\nu_u(\diffns a)$ is continuous, adapted and with bounded variations, Tanaka's formula yields
	\begin{align*}
		X^1_t\vee X_t^2 &= X_t^2 + (X_t^1 - X_t^2)^+
		            = X_t^2 + \int_0^t1_{\{X^1_u > X^2_u\}}\diffns  (X^1_u - X^2_u) + \frac{1}{2}L^0_t(X^1 - X^2)\\
		            &= X_0 + \int_0^t1_{\{X^1_u > X^2_u\}}\diffns X_u^1 + \int_0^t1_{\{X^1_u\le X^2_u\}}\diffns X^2_u\\
		            &= X_0 + \int_0^t\sigma_u(X^1_u\vee X_u^2 )\diffns W_u + \int_0^t\int_\mathbb{R}\nu_u(\diffns a)\left(1_{\{X^1_u>X^2_u\}}\diffns L^a_u(X^1) + 1_{\{X^1_u\le X^2_u\}}\diffns L^a_u(X^2)\right).
	\end{align*}
	By (a trivial adaptation of) the result of \cite{Ouknine90} on the local time of the maximum, one has
	\begin{equation*}
		L^a_t(X^1\vee X^2) = \int_0^t1_{\{X^1_u> X_u^2\}}\diffns L_u^a(X^1) + \int_0^t1_{\{X^1_u\le X^2_u\}}\diffns L_u^a(X^2).
	\end{equation*}
	Hence,
	\begin{equation*}
		X^1_t\vee X^2_t = \int_0^t\sigma_u(X^1_u\vee X_u^2)\diffns W_u + \int_0^t\int_\mathbb{R}\nu_u(\diffns a)\diffns L^a_u(X^1\vee X^2).
	\end{equation*}
	Using the identity $L^a(X^1\wedge X^2) = L^a(X^1) + L^a(X^2) - L^a(X^1\vee X^2)$ (see e.g. \cite{Ouknine90}), the argument above also shows that $X^1\wedge X^2$ is a solution.
	Thus, by uniqueness in law, $X^1\wedge X^2$ and $X^1\vee X^2$ must have the same law.
	Therefore, for every $t \in [0,T]$,
	\begin{equation*}
		E[|X^1_t - X^2_t|] = E[X^1_t\vee X^2_t] - E[X^1_t\wedge X^2_t] = 0.
	\end{equation*}
	That is, $X^1$ and $X^2$ are indistinguishable since they are continuous processes.
\end{proof}
The condition (LT) is standard in the study of time inhomogeneous SDEs, see e.g. \cite{Ben-Bou-Ouk13} and \cite{Bar-Per84}.
We present below an example of functions satisfying (LT).
\begin{example}
	Let $d=1$ and $\sigma:[0,T]\times \Omega\times \mathbb{R} \to \mathbb{R}$ be  such that there is $\varepsilon>0$: $\sigma \ge \varepsilon$ and there are two functions $\alpha^1, \alpha^2:[0,T]\times \Omega\times \mathbb{R} \to \mathbb{R}$ that are increasing in $x$ for all $t$ and of uniformly bounded variation in $t$ on every compact of $\mathbb{R}$, and such that $1/\sigma = \alpha^1 - \alpha^2$.
	Pathwise uniqueness for SDEs with diffusion coefficients satisfying these conditions were first studied in \cite{Nakao}.
In fact, set $Y^i_t:= F(t, X^i_t)$, with $F(t,x):= \int_0^x\frac{\diffns u}{\sigma_t(u)}$, where $X^i$ are processes satisfying \eqref{eq:lt}.
	It follows by \cite[Theorem 1]{Ouknine_fonc_89} that for each $i=1,2$, there is a continuous bounded variation process $ V^i$ such that $Y^i_t = B_t + V^i_t$.
	Thus, $Y^1_t - Y^2_t = V^1_t - V^2_t$.
	Since the right hand side of the latter equality is a continuous process with bounded variations, it holds $L^0_t(Y^1 - Y^2) =0$.
	Thus, by \cite{Ouknine88} one has $L^0_t(X^1 - X^2)=0$.
\end{example}
Next, we assume that the flow of measures $(\nu_t)_{t\in [0,T]}$ is constant, i.e. for all $t$, one has $\nu_t \equiv \nu$. Then the SDE \eqref{eq:SDE_local_nonhom} becomes
\begin{equation}
\label{eq:SDE.inhom}
	X_t = x + \int_0^t\sigma_u(X_u)\diffns W_u + \int_{\mathbb{R}}L_t^a(X)\,\nu(\diffns a).
\end{equation}
In this case, the requirement on the uniqueness in law in Theorem \ref{thm:lt_local_nonhom} can be dropped.
The SDE \eqref{eq:SDE.inhom} has been considered in \cite{LeGall84} and subsequently in \cite{Engl-Schm123} and \cite{Blei-Engl}, under the conditions
\begin{itemize}
	\item[(A1)] $|\nu|(\{a\})<1$ for all $a \in \mathbb{R}$,
	\item[(A2)] $|\nu|(\mathbb{R})<\infty$.
\end{itemize}
We show in Proposition \ref{pro:lt} below that the above conditions can be weakened.
Consider the functions
	\begin{equation*}
		f_\nu(x):= \exp(-2\nu^c((-\infty,x]))\Pi_{y\le x}\left(\frac{1-\nu\{y\}}{1+\nu\{y\}} \right)\quad \text{and} \quad F_\nu(x):= \int_0^x f_\nu(z)\diffns z,
	\end{equation*}
where $\nu^c$ denotes the continuous part of the measure $\nu$.
	Recall that due to conditions (A1) and (A2), the function $f_\nu$ is well-defined, increasing, right-continuous and $0<\ubar{m}\le f_\nu\le \bar{m}$ for some $\ubar{m},\bar{m}\in \mathbb{R}$.
	Furthermore, it can be checked that $F_\nu$ is invertible, and $F_\nu$ and $F_\nu^{-1}$ are Lipschitz continuous functions; see e.g. \cite{LeGall84} for details.
	Denote by $N_\sigma$ the set of zeros of the function $\sigma$ defined as follows
	$N_\sigma:= \{x \in \mathbb{R}: \sigma_t(x)=0\,\,\,  \diffns t\text{-a.s.}\}$.
\begin{proposition}
\label{pro:lt}
	Suppose $|\nu|(\{a\})<1$ for all $a \in N_\sigma$ and $|\nu|(N_\sigma^c)<\infty$. In addition, suppose $\sigma$ satisfies condition (LT). Then the SDE \eqref{eq:SDE.inhom} satisfies the pathwise uniqueness property.
\end{proposition}
The proof of Proposition \ref{pro:lt} uses the following lemma that gives conditions under which two continuous processes are indistinguishable:
\begin{lemma}
\label{lem:indist}
	Let $\nu$ be a measure satisfying conditions \textup{(}A1\textup{)} and \textup{(}A2\textup{)} and
	let $X^1$ and $X^2$ be two semimartingales of the form
	\begin{equation*}
		X^i = x + M_t^i + \int_\mathbb{R}L^a_t(X^i)\nu(\diffns a), \quad i=1,2,
	\end{equation*}
	where $M^i$ are continuous local martingales.
	If $L^0(X^1 - X^2)=0$, then $X^1$ and $X^2$ are indistinguishable.
\end{lemma}
\begin{proof}
	First recall that the function $F_\nu$ satisfies $\ubar{m}(x - y)^+\le (F_\nu(x) - F_\nu(y))^+\le \bar{m}(x - y)^+$ for all $x, y$; see e.g. \cite{Ben-Bou-Ouk13}.
	Set $Y^i_t: = F_\nu(X^i_t)$, $i=1,2$, $t \in [0,T]$.
	It follows from Tanaka's formula that
	\begin{equation*}
		Y^{i}_t = F_\nu(x) + \int_0^tf_\nu(X^i_u)Z^i_u\diffns W_u,
	\end{equation*}
	with $Z^i$ the predictable process such that $M^i_t = M^i_0 + \int_0^tZ^i_u\diffns W_u$.
	Thus, $F_\nu(X^i)$ is a local martingale for each $i$.
	Moreover, $\ubar{m}(X^1 - X^2)^+ \le (Y^1 - Y^2)^+\le \bar{m}(X^1 - X^2)^+$, so that since $L^0(X^1 - X^2)=0$, it follows from the comparison theorem for local times (see \cite{Ouknine88}) that $L^0_t(Y^1 - Y^2) =0$.
	Thus, an application of Tanaka's formula again shows that
	\begin{equation*}
		|Y^1_t - Y^2_t| = \int_0^t\textrm{sign}(Y^1_u - Y^2_u)\diffns  (Y^1_u - Y^2_u),
	\end{equation*}
	from which a simple localization argument shows that $E[|Y^1_t - Y^2_t|] =0$, i.e. $Y^1_t = Y^2_t$.
	Since $F_\nu$ is invertible, this implies $X^1_t = X^2_t$ for every $t \in [0,T]$.
	We therefore conclude that $X^1$ and $X^2$ are indistinguishable since they are continuous processes.
\end{proof}
We now turn to the proof of Proposition \ref{pro:lt}.
\begin{proof}[of Proposition \ref{pro:lt}]
First notice that 
\begin{equation}
\label{eq:lemLT}
	L^x_t(X)=0 \quad \text{for all $x \in N_\sigma$ and every solution $X$ of \eqref{eq:SDE.inhom}}.
\end{equation}
Indeed, let $x \in N_\sigma$.
	Since $L^a_t(x)=0$ for all $(t,\omega)\in [0,T]\times \Omega$ and $a \in \mathbb{R}$, the (constant) process $x$ solves the SDE \eqref{eq:SDE.inhom} with initial condition $X_0 = x$.
	Thus, for every solution $X$ of \eqref{eq:SDE.inhom}, it follows from condition (LT) that $L^x_t(X) = L^0_t(X-x)=0$.

	Let $\tilde{\nu}$ denote the restriction of the measure $\nu$ on $N_\sigma^c$, i.e.	$\tilde \nu (A):= \nu(A \cap N_\sigma^c)$ for all Borel subset $A$ of $\mathbb{R}$.
	The measure $\tilde \nu$ satisfies the conditions (A1) and (A2) and by \eqref{eq:lemLT}, every solution $X$ of the SDE \eqref{eq:SDE.inhom} satisfies
	\begin{equation*}
		X_t = x + \int_0^t\sigma_u(X_u)\diffns W_u + \int_{\mathbb{R}}L_t^a(X)\,\tilde{\nu}(\diffns a).
	\end{equation*}
	Thus, the result follows by Lemma \ref{lem:indist}.
\end{proof}

\subsection{The time homogeneous case with deterministic coefficient}
In this subsection, we study the SDE
\begin{equation}
\label{eq:SDE.homogene}
	X_t = x + \int_0^t\sigma(X_u)\,dW_u + \int_{\mathbb{R}}L^a(X)\nu(da),
\end{equation}
with $\sigma:\mathbb{R}\to \mathbb{R}^d$ a measurable function.
We show that in this case, the condition (LT) can essentially be replaced by integrability conditions on $\sigma$ to obtain pathwise uniqueness.

Consider the following conditions: There exist two functions $f:\mathbb{R}\to \mathbb{R}_+$ and $h:\mathbb{R}\to \mathbb{R}_+$ such that
\begin{itemize}
	\item[(A3)] $\int_{0^+}\frac{\diffns a}{h^2(a)}=+\infty$ and $f/\sigma \in L^2_{\text{loc}}(\mathbb{R})$,
	\item[(A4)] $|\sigma(x) - \sigma(y)| \le (f(x) + f(y))h(|x - y|)\quad \text{and} \quad N_\sigma \subseteq N_f:= \{x:f(x) = 0\}$.
\end{itemize}

\begin{theorem}
\label{thm:pathwise.unique}
	Assume that  the conditions \textup{(}A1\textup{)}-\textup{(}A4\textup{)} are satisfied.
	Then, the SDE \eqref{eq:SDE.homogene} has the pathwise uniqueness property.
\end{theorem}
The proof of Theorem \ref{thm:pathwise.unique} uses the following lemma:

\begin{lemma}
\label{lem:int.zero}
	Let $X^1$ and $X^2$ be two solutions of \eqref{eq:SDE.homogene}.
	Suppose $\int_0^\cdot\frac{\diffns \langle X^1_u - X^2_u\rangle}{h^2(X^1_u - X^2_u)}<\infty$. Then $L^0(X^1-X^2)=0$.
\end{lemma}
\begin{proof}
	Assume by contradiction that there is $t \in [0,T]$, $\delta>0$ and a set $A \in {\cal F}$ with $P(A)>0$ such that $L_t^0(X^1 - X^2)>\delta$ on $A$.
	Since the function $a\mapsto L^a_t(X^1 - X^2)$ is right continuous, there is $\varepsilon >0$ such that $L^a_t(X^1 - X^2)(\omega)>\delta/2$ for $\omega \in A$.
	Thus, by (A3) one has
	\begin{equation*}
		+\infty =\frac{\delta}{2}\int_0^\varepsilon\frac{1}{h^2(a)}\diffns a \le\int_\mathbb{R}\frac{L^a_t(X^1 - X^2)}{h^2(a)}\diffns a = \int_0^t\frac{\diffns \langle X^1_u - X^2_u\rangle}{h^2(X^1_u - X^2_u)}<+\infty \quad \text{on } A,
	\end{equation*}
	where the second equality follows from the occupation time formula, see e.g. \cite{Revuz1999}.
	Thus $P(A)=0$, which is a contradiction.
	Therefore, $L^0(X^1-X^2)=0$.
\end{proof}

\begin{proof}[of Theorem \ref{thm:pathwise.unique}]
	In light of lemmas \ref{lem:indist} and \ref{lem:int.zero}, it remains to show that $\int_0^\cdot\frac{\diffns \langle X^1_u - X^2_u\rangle}{h(X^1_u - X^2_u)}<\infty$.
	This follows again as an application of the occupation time formula.
	In fact, by (A4) one has
	\begin{align*}
		\int_0^t\frac{\diffns \langle X^1_u - X^2_u\rangle}{h^2(X^1_u - X^2_u)} &= \int_0^t\frac{(\sigma(X^1_u) - \sigma(X^2_u))^2}{h^2(X^1_u-X^2_u)}\diffns u \le \int_0^t\left(f(X^1_u) + f(X^2_u)\right)^2\diffns u\\
		& \le 2\int_0^t\frac{f^2(X^1_u) }{\sigma^2(X^1_u)}\sigma^2(X^1_u)1_{\{X^1_u \notin N_f\}}\diffns u + 2\int_0^t\frac{f^2(X^2_u)}{\sigma^2(X^2_u)}\sigma^2(X^2_u)1_{\{X^2_u \notin N_f\}}\diffns u\\
		&= 2\int_{\mathbb{R}}\frac{f^2(a)}{\sigma^2(a)}\,L^a_t(X^1)1_{N_f^c}\diffns a + 2\int_{\mathbb{R}}\frac{f^2(a)}{\sigma^2(a)}\,L^a_t(X^2)1_{N_f^c}\diffns a,
	\end{align*}
	where the last equality follows from the occupation time formula.
	Let $(t,\omega) \in [0,T]\times \Omega$.
	Since the function $a\mapsto L^a_t(X^i)(\omega)$, $i=1,2$ has support on the compact $K^i_t(\omega):= [\inf_{0\le u\le t}X_u^i(\omega), \sup_{0\le u\le t}X^i_u(\omega)]$, due to (A3) it holds
	\begin{equation*}
		\int_\mathbb{R}\frac{f^2(a)}{\sigma^2(a)}L^a_t(X^i)(\omega)1_{N_f^c}\diffns a  \le \sup_{a \in K^i_t(\omega)}L^a_t(X^i)\int_{K^i_t(\omega)\cap N^c_f}\frac{f^2(a)}{\sigma^2(a)}\diffns a<\infty.
	\end{equation*}
	This concludes the proof.
\end{proof}
\begin{remark}
	Let us observe that the result of Theorem \ref{thm:pathwise.unique} is known when $\sigma^2>0$; see e.g. \cite{Engl-Schm123} and \cite{Cham-Jabin17}. 
	We allow $\sigma$ to vanish and, our proof is based on different and more direct arguments.
\end{remark}

A particularly interesting example where the conditions (A3) and (A4) are fulfilled arises when $\sigma$ belongs to a suitable Sobolev space.
In fact, let us consider the maximal operator ${\cal M}$ of a function $f:\mathbb{R}\to\mathbb{R}^d$ defined as
\begin{equation*}
	{\cal M}f(x):=\sup_{r>0}\frac{1}{B_r}\int_{B_r}|f|(x+\diffns z),\quad x\in \mathbb{R}^d,
\end{equation*}
where $B_r$ is the ball of radius $r$ around the origin, and the derivative operator
\begin{equation*}
	\partial_x^{1/2}f:= \mathscr{F}^{-1}|z|^{1/2}\mathscr{F}f,
\end{equation*}
with $\mathscr{F}$ the Fourier transform in $\mathbb{R}^d$.
The function ${\cal M}f$ is positive and Borel measurable; see for example \cite{Stein70}.
Hence its integral with respect to a Borel measure is well-defined, with value in $\mathbb{R}_+\cup\{+\infty\}$.
Moreover, for any locally integrable function $f$, the derivative $\partial_x^{1/2}f$ is well-defined.

\begin{corollary}
\label{cor:pathwise.uniqueness}
	Assume that conditions (A1)-(A2) are satisfied. Assume further  that $\partial^{1/2}_x\sigma$ is a locally finite Radon measure and
	\[
	{\cal M}\partial^{1/2}_x\sigma/\sigma \in L^2_{\text{loc}}(\mathbb{R}),\quad \sigma \in L^1_{\text{loc}}(\mathbb{R})\quad \text{and} \quad N_\sigma \subseteq \{x: {\cal M}\partial^{1/2}_x\sigma(x)=0\}.
	\]
	Then the SDE \eqref{eq:SDE.homogene} satisfies the pathwise uniqueness property.
\end{corollary}
\begin{proof}
It follows from \cite[Lemma 3.5]{Cham-Jabin17} that  $\sigma$ satisfies
	\begin{equation*}
		|\sigma(x) - \sigma(y)|\le \left({\cal M}\partial^{1/2}_x\sigma(x) + {\cal M}\partial^{1/2}_x\sigma(y) \right)|x-y|^{1/2}.
	\end{equation*}
Thus, the functions $h(x):=|x|^{1/2}$ and $f(x):= {\cal M}\partial_x^{1/2}\sigma(x)$ satisfy the conditions (A3) and (A4).
The result follows from Theorem \ref{thm:pathwise.unique}.
\end{proof}
\begin{remark}
In fact, the preceding corollary extends recent results by \cite{Cham-Jabin17} and in particular \cite[Corollary 2.22]{Cham-Jabin17}.
\end{remark}
\subsection{Reflected SDEs}
In this subsection, we consider the reflected SDE
\begin{equation}
\label{eq:reflected.sde}
\begin{cases}
	X_t  = x + \int_0^t\sigma_u(X_u)\diffns W_u + \frac{1}{2}L^0_t(X),\\
	  X\ge 0.\end{cases}
\end{equation}
As in the previous sections, we are interested in the pathwise uniqueness property of the above equation.
\begin{proposition}
\label{prop:reflected}
	If we have uniqueness in law and given any two solutions $X$ and $Y$ of \eqref{eq:reflected.sde} the measure $\diffns L^0_t(X - Y)$ is supported by the set $\{X = Y =0\}$, then the SDE \eqref{eq:reflected.sde} satisfies the pathwise uniqueness property.
\end{proposition}
\begin{proof}
	It follows from \cite[Lemma 1]{Ouknine93} that for every integer $n\ge 1$, and for any continuous semimartingales $X$ and $Y$, it holds
	 \begin{equation}
	 \label{eq:2n+1}
	 	L^0_t(X^{2n+1} - Y^{2n+1}) = (2n+1)\int_0^t(X^{2n}_s + Y^{2n}_s)\diffns L_s^0(X - Y).
	 \end{equation}
	In particular, for $n =1$,
		 \begin{equation}
	 \label{eq:2n+1}
	 	L^0_t(X^{3} - Y^{3}) = 3\int_0^t(X^{2}_s + Y^{2}_s)\diffns L_s^0(X - Y).
	 \end{equation}
	 Thus, if $\text{supp}(dL^0_s(X - Y))\subseteq \{X = Y = 0\}$, it holds $L^0_t(X^{3} - Y^{3}) =0$.
	 Let $X$ and $Y$ be two solutions of \eqref{eq:reflected.sde}.
	 Then, as shown in \cite{Ouknine93}, $X^{3}$ and $Y^{3}$ satisfy the same SDE, so that $X^{3}\vee Y^{3}$ and $X^{3}\wedge Y^{3}$ are also solutions of \eqref{eq:reflected.sde} (confer the proof of Theorem \ref{thm:lt_local_nonhom}.)
	 But $X^{3}, Y^{3}$ and $X^{3}\vee Y^{3}$ have the same law.
	Thus, $X = Y.$
\end{proof}
The subsequent example provides a diffusion for which the above result is valid.
Notice that $\sigma$ is not necessarily Lipschitz continuous.
\begin{example}
\label{exa:support}
	Assume that there is an integer $n$ such that 
	\begin{equation*}
		|x^{2n}\sigma_t(x) - y^{2n}\sigma_t(y)|^2 \le C|x^{2n+1} - y^{2n+1}|
	\end{equation*}
	for some $C\ge0$.
	Thus, by the arguments in the proof of the main result of \cite{Ouknine93}, it holds $L^0_t(X^{2n+1} - Y^{2n+1})=0$.
	It then follows from \eqref{eq:2n+1} that the measure $\diffns L^0_s(X^1 - X^2)$ is supported by the set $\{X^1 = X^2 = 0\}$ whenever $X^1$ and $X^2$ are two solutions of \eqref{eq:reflected.sde}.
	In fact, 
	if $A:= \{ (t,\omega): (X^1_t)^2 + (X^2_t)^2>0\} $, then we have $A=\cup A_n$ with $A_n = \{ (X^1)^2 + (X^2)^2>\frac 1n \}$.
  This shows that
  \begin{equation*}
    0\le \int_{A_n}dL^0_s(X^1-X^2)\le n \int ((X^1_s)^2 + (X^2_s)^2)dL^0_s(X^1-X^2) = 0.
  \end{equation*}
By monotone convergence, this implies $\int_AdL^0_s(X^1-X^2) = 0$. Thus, the support of the measure $dL^0_s(X^1-X^2)$ is in $A^c = \{ X^1 = X^2 = 0 \}$.
\end{example}

A similar method allows us to study the case of SDEs with jumps.
In fact, consider the SDE
\begin{equation}
\label{eq:reflexed-jump}
X_t = x + \int_0^t\sigma_u(X_u)\diffns W_u + \int_{[0,t)}b(X_{u-})\diffns u + \int_0^t\int_{\mathbb{R}\setminus\{0\}}\gamma(X_{u-},z)\tilde\eta(\diffns u, \diffns z) + \frac{1}{2}L^0_t(X),\quad X\ge0
\end{equation}
where $b:\mathbb{R}\to \mathbb{R}$ is bounded and measurable, $\gamma:\mathbb{R}\times \mathbb{R}\to \mathbb{R}$ is bounded and measurable and $\tilde\eta$ a given signed measure on $[0,T]\times \mathbb{R}$.
\begin{proposition}
\label{prop:jump}
Suppose the following:
\begin{itemize}
	\item[(i)] Uniqueness in law holds,
	\item[(ii)] the function $x\mapsto \gamma(x,z)+x$ is increasing, $\eta(dz)$-a.e in $\{|z|<\varepsilon\}$ for some $\varepsilon>0$,
	\item[(iii)] there is an odd number $n= 2k+1$, $k \in \mathbb{N}$ and a constant $C\ge0$ such that
	\begin{equation*}
	|(x^n\sigma_t(x) - y^n\sigma_t(y))|+ |\int_{\mathbb{R}\setminus \{0\}}(x^n \gamma(x,z) - y^n\gamma(y,z))^2\eta(\diffns z)| \le C |x^{n+1}- y^{n+1}|.
	\end{equation*} 
\end{itemize}
Then the SDE \eqref{eq:reflexed-jump} has the pathwise uniqueness property.
\end{proposition}
\begin{proof}
First notice that for every two (strong) solutions $X^1$ and $X^2$, the measure $dL^0_t(X^1 - X^2)$ is supported by the set $\{X^1_- = X^2_- =0\}$.
In fact, this follows as in the case $\gamma=0$ studied in Example \ref{exa:support}, using occupation time formula.
Indeed, since 
\begin{equation*}
	[X^1 - X^2]^c_t = \int_0^t(\sigma_u(X^1_u)-\sigma_u(X^2_u))^2\diffns u + \int_0^t\int_{\mathbb{R}\setminus\{0\}}(\gamma(X^1_{u-},z) - \gamma(X^2_{u-},z))^2\eta(\diffns z)\diffns s,
\end{equation*}
it follows that for every measurable function $f:\mathbb{R}\to \mathbb{R}$, we have
\begin{equation*}
	\int_\mathbb{R}f(a)L^a_t(X^1 - X^2)\diffns a = \int_0^tf(X^1_u-X^2_u)\Big((\sigma_u(X^1_u)-\sigma_u(X^2_u))^2 + \int_{\mathbb{R}\setminus\{0\}}(\gamma(X^1_{u-},z) - \gamma(X^2_{u-},z))^2\eta(\diffns z)\Big) \diffns u.
\end{equation*}
Now, let $X^1$ and $X^2$ be two solutions of \eqref{eq:reflexed-jump}.
Using Tanaka's formula as in the proof of Theorem \ref{thm:lt_local_nonhom}, it holds
\begin{align*}
	X^1\vee X^2
			&= x + \int_0^t\sigma_u(X^1_u\vee X^2_u)\diffns W_u + \int_0^tb(X_u^1\vee X^2_u)\diffns u + \int_0^t\int_{\mathbb{R}\setminus\{0\}}\gamma((X^1\vee X^2)_{u-},z)\tilde\eta(\diffns s,\diffns z)\\
			& + \frac{1}{2}\Big(\int_0^t1_{\{X^1_->X^2_-\}}\diffns L^1 + \int_0^t1_{\{X^1_-\le X^2_-\}}\diffns L^2 + \int_0^t1_{\{X^1_-=X^2_- = 0\}}\diffns L^0_s(X^1-X^2)\Big).
\end{align*}
 This last expression is exactly $\frac{1}{2}L^0_t(X^1\vee X^2)$ as a consequence of \cite{Ouknine90} and the fact that the measure $\diffns L^0_t(X^1 - X^2)$ is supported by the set $\{X^1_- = X^2_- =0\}$.
 The result now follows as in Proposition \ref{prop:reflected} by the uniqueness in law.
\end{proof}
\begin{remark}
Suppose that there exists a monotone function $f$ such that
\begin{equation*}
(\sigma_t(x) - \sigma_t(y))^2 + \int_{\mathbb{R}\setminus \{0\}}(\gamma(x,z) - \gamma(y,z))^2\eta(\diffns z) \le |f(x) - f(y)|,\quad \text{and } \sigma >\varepsilon.
\end{equation*}
Then $L^0_t(X^1-X^2) = 0$ whenever $X^1$ and $X^2$ are solutions.

Moreover, if $\tilde\eta(\diffns s,\diffns z) = N(\diffns t,\diffns z) - \eta(\diffns z)\diffns t$, then $N(t,A) = N([0,t]\times A)$ is a Poisson process with intensity $t\eta(A)$, where $0\notin \bar{A}$, the closure of $A$, then in Equation \eqref{eq:reflexed-jump}, we do not need $\mathbb{R}\setminus \{0\}$ but just a neighborhood of $0$, i.e. $\{0< |z|<\varepsilon\}$.
\end{remark}

\section{Applications}
\label{sec:applications}

In the remainder of the paper, we apply the pathwise uniqueness results of the previous section to the theory of SDEs with and without local time of the unknown.
Most of our proofs will use the well-known Zvonkin's transform already introduced in Lemma \ref{lem:indist}.
Thus, the function
$$\tilde\sigma(x):= (f_\nu\cdot \sigma)\circ F^{-1}_\nu(x)$$
will play a central role in our arguments.
\subsection{Existence}
\label{subsec:exist}

In this section, we establish existence and uniqueness of strong solution for the SDE \eqref{eq:SDE_local_nonhom}, when the measures $\nu$ is of a specific type. More precisely, assume that the measure $\nu$ has the following form: 
\begin{equation*}
	\nu(\diffns a) := \sum_{i=1}^n\alpha^i\delta_{a^i},
\end{equation*}
with $\alpha^i \in \mathbb{R}$ and $a^1< a^2<\ldots< a^n$.

Let us now set $\beta^i_t:=\alpha^i\alpha_t(a^i)$, so that the SDE \eqref{eq:SDE_local_nonhom} becomes 
\begin{equation}
\label{eq:SDE.etore}
	X_t = x + \int_0^t\sigma_u(X_u)\diffns W_u + \sum_{i = 1}^n\int_0^t\beta^i(u)\diffns L^{a^i}_u(X).
\end{equation}
The above SDE \eqref{eq:SDE.etore} with an additional drift term was recently studied in \cite{Etore2017}.
For the case $\sigma = 1$, $n=1$ and $a^1 = 0$, we refer to \cite{Weinryb}. The next result generalises \cite[Theorem 3.5]{Etore2017}. 
\begin{proposition}
\label{thm:etore}
	Assume 
	there is $m, M \in \mathbb{R}_+$ such that $0<m\le \sigma\le M$ and that $\sigma$ satisfies the condition (LT), and for every $i = 1,\dots, n$, $\beta^i:[0,T]\to [\ubar{k}, \bar{k}]$ is of class $C^1$, with $-1< \ubar{k}<\bar{k}<1 $  and $|(\beta_t^i)'|\le M$ for all $t \in [0,T]$.
	Then, the time inhomogeneous SDE \eqref{eq:SDE.etore} has a unique strong solution.
\end{proposition}
\begin{proof}
It follows from \cite{Etore2017} that we have uniqueness in law.
Since $\sigma$ additionally satisfies (LT), it follows from Theorem \ref{thm:lt_local_nonhom} that the SDE satisfies the pathwise uniqueness property.
The existence of a unique strong solution therefore follows by \citet{Yam-Wata}.
\end{proof}
\begin{remark}
As pointed above a similar result was obtained in \cite[Theorem 3.5]{Etore2017}. 
Notice however that in the latter work, the authors needed smoothness of the diffusion coefficient, we relax this assumption in this paper in the sense that, we only require $\sigma$ to satisfy (LT). In addition, our proof is based on the comparison theorem for local times.
\end{remark}

Let us now turn to the time-inhomogeneous case.
Put 
$$I_\sigma:= \left\{a\in \mathbb{R}: \int_{O} \sigma^{-2}(y)\diffns y<\infty \text{ for all open neighborhood $O$ of $a$}\right\} .$$
\begin{proposition}
\label{pro:exist_lt}
 	Under the assumptions of Theorem \ref{thm:pathwise.unique}, the SDE \eqref{eq:SDE.homogene} admits a unique (non-constant) strong solution if and only if $N_\sigma^c\subseteq I_\sigma$.
 \end{proposition} 
 \begin{proof}
As in the proof of Lemma \ref{lem:indist}, the transformation $Y:= F_\nu(X)$ satisfies the dynamics
 \begin{equation}
 \label{eq:sde.transf}
 	Y_t = F_\nu(x) + \int_0^t\tilde{\sigma}(Y_u)\diffns W_u.
 \end{equation}
 Since $0<\ubar{m}\le f_\nu\le\bar{m}<\infty$,  it can be checked that $F_\nu(N_\sigma) = N_{\tilde{\sigma}}$ and $F_{\nu}(I_\sigma) = I_{\tilde\sigma}$.
 Thus, since $F_\nu$ is invertible, it holds $F_\nu(N_\sigma)^c = F_\nu(N_\sigma^c) \subseteq F_\nu(I_\sigma)$.
 This implies that the function $\tilde{\sigma}^{-2}1_{N^c_{\tilde{\sigma}}}$ is locally integrable.
 Hence, by \cite[Theorem 2.2]{Eng-Sch85} the SDE \eqref{eq:sde.transf} admits a weak solution.
 The result follows now by Theorem \ref{thm:pathwise.unique} and the Yamada-Watanabe theorem.

 Reciprocally, if the SDE \eqref{eq:SDE.homogene} admits a non-constant solution, then the SDE \eqref{eq:sde.transf} admits a solution as well.
 Thus, it follows by \cite[Theorem 4.7]{Engl-Schm123} that $N^c_{\tilde\sigma}\subseteq I_{\tilde\sigma}$.
\end{proof}

\subsection{Path and space regularity}
In this section we study continuity of the solution of the SDE \eqref{eq:SDE.homogene} in times as well as w.r.t. the initial condition.
Let $\alpha>0$ and denote by $C^\alpha([0,T],\mathbb{R})$ the space of $\alpha$-H\"older continuous functions on $[0,T]$ with values in $\mathbb{R}$.
Its norm is defined by
\begin{equation*}
	||f||_\alpha:= \sup_{0\le t\le T}|f(t)| + \sup_{0\le s\le t\le T}\frac{|f(s) - f(t)|}{|s-t|^\alpha}. 
\end{equation*}
Let us denote by $X^x$ the solution of the SDE \eqref{eq:SDE.homogene} with initial condition $x$.

\begin{proposition}
\label{pro:cont}
	Assume that \textup{(}A1\textup{)}-\textup{(}A4\textup{)} hold and that  $\sigma$ is continuous and of linear growth, i.e. there is a constant $A\ge 0$ such that $|\sigma(x)|\le A(1 +|x|)$ for all $x \in \mathbb{R}$.
	Then, for all $\alpha \in [0,1/2)$ and $\varepsilon>0$, one has
	\begin{equation}
		\lim_{x \to x_0}P\left(||X^x - X^{x_0}||_\alpha >\varepsilon\right) = 0 \quad\text{for all } x_0 \in \mathbb{R}.
	\end{equation}	
\end{proposition}
\begin{proof}
	For each $x \in \mathbb{R}$ the transformation $Y^x:= F_\nu(X^x)$ satisfies
 	\begin{equation}
	\label{eq:Y.SDE}
 		Y^x_t = F_\nu(x) + \int_0^t\tilde{\sigma}(Y_u^x)\diffns W_u.
 	\end{equation}
	Since $f_\nu$ is bounded, $\sigma$ of linear growth and $F^{-1}_\nu$ Lipschitz continuous, one has
	\begin{equation*}
		|\tilde\sigma(y)|=|f_\mu(F^{-1}_\nu(y))\sigma(F^{-1}_\nu(y)) |\le \bar{m}A(1 + |F^{-1}_\nu(y)|)\le \bar{m}C(1 + F^{-1}_\nu(0) + \bar{m}|y|).
	\end{equation*}
	That is, the function $\tilde{\sigma}$ is continuous and of linear growth.
	Thus, since by Theorem \ref{thm:pathwise.unique}  pathwise uniqueness holds for the SDE \eqref{eq:Y.SDE}, it follows from \cite[Proposition 3.8]{Bah-Mer-Ouk98} that
	\begin{equation*}
		\lim_{x \to x_0}P\left(||Y^x - Y^{x_0}||_\alpha>\varepsilon\right) = 0.
	\end{equation*}
	By Lipschitz continuity of $F^{-1}_\nu$, it holds
	\begin{equation*}
		P\left(||X^x - X^{x_0}||_\alpha >\varepsilon\right) = P\left(||F^{-1}_\nu(Y^x) - F^{-1}_\nu(Y^{x_0})||_\alpha >\varepsilon\right)\le P\left(||Y^x - Y^{x_0}||_\alpha >\varepsilon\right).
	\end{equation*}
	This shows the desired result.
\end{proof}
The solution of the SDE \eqref{eq:SDE.homogene} is also H\"older continuous in time. The proof of the result is omitted.
\begin{proposition}
	Assume that conditions \textup{(}A1\textup{)}-\textup{(}A4\textup{)} are satisfied.
	If $\sigma$ is locally integrable, then there is a constant $C>0$ such that
	\begin{equation*}
	 	E\left[|X_t - X_s|^2\right] \le C|t-s|^{1/2}\quad \text{for all}\quad s,t \in [0,T].
	 \end{equation*} 
\end{proposition}

\subsection{Feynman-Kac type formula}
Let $(s,x) \in [0,T]\times \mathbb{R}$ and $X^{s,x}$ be the solution (if it exists) of the SDE
\begin{equation*}
	X^{s,x}_t = x +\int_s^t\sigma(X^{s,x}_u)\diffns W_u + \int_\mathbb{R}L^a_t(X^{s,x})\nu(\diffns a).
\end{equation*}
\begin{proposition}
	Assume that the conditions \textup{(}A1\textup{)}-\textup{(}A4\textup{)} are satisfied.
	Let $f:\mathbb{R}\to \mathbb{R}$ and $g:[0,T]\times \mathbb{R}\to \mathbb{R}$ be two bounded continuous functions.
	Then the function 
	\begin{equation*}
		v(s,x) := E\Big[f(X^{s,x}_T) + \int_s^Tg_u(X^{s,x}_u)\diffns u \Big]
	\end{equation*}
	satisfies $v(s,x)=u(s, F_\nu(x))$, where $u$ is a viscosity solution of the partial differential equation
	\begin{equation}
	\label{eq:pde}
		\begin{cases}
			&\partial_tu + \frac{1}{2}((f_\nu\cdot \sigma)\circ F^{-1}_\nu)^2\partial^2_{xx}u + g\circ F^{-1}_\nu =0\\
			&u(T,x) = f\circ F^{-1}_\nu(x).
		\end{cases}
	\end{equation}
\end{proposition}
\begin{proof}
	Putting $y:= F_\nu(x)$ and $Y^{s,y}:= F_\nu(X^{s,x})$, we have  
	\begin{equation}
 	Y_t^{s,y} = y + \int_s^t\tilde{\sigma}(Y_u^{s,y})\diffns W_u
 	\end{equation}
 	and $v(s,x) = E[f(F^{-1}_\nu(Y^{s,y}_T)) + \int_s^Tg_u(F^{-1}_\nu(Y_u^{s,y}))\diffns u] = :u(s,y)$.
 	Thus, $v(s,x) = u(s, F_\nu(x))$.
 	It remains to show that $u$ is a viscosity solution of \eqref{eq:pde}.
	Let $(s,y)\in [0,T]\times \mathbb{R}$ and $\varphi \in C^{1,2}$ be a test function with bounded derivatives such that $\varphi - u$ attains a global maximum at $(s,y)$ with $\varphi(s,y) = u(s,y)$.	
	If $s = T$, we clearly have $\varphi(T,y) = f\circ F^{-1}_\nu(y)$.
 	Thus, for all $\varepsilon>0$, It\^o's formula yields
 	\begin{align*}
 		\varphi(s + \varepsilon, Y^{s,y}_{s + \varepsilon}) - \varphi(s, y) = \int_s^{s + \varepsilon}\Big\{\partial_t\varphi(t,Y^{s,y}_t) + \frac{1}{2}\tilde{\sigma}^2(Y^{s,y}_t)\partial_{xx}^2\varphi(s, Y^{s,y}_t)\Big\}\diffns t + \int_s^{s + \varepsilon}\partial_x\varphi(t, Y^{s,y}_t)\tilde{\sigma}(Y^{s,y}_t)\diffns W_t
 \end{align*}
 Thus, taking expectation on both sides, one has
 \begin{multline*}
 	E\Big[u(s+\varepsilon, Y^{s,y}_{s + \varepsilon})+\int_s^{s+\varepsilon}g_t( F^{-1}_\nu(Y^{s,y}_t))\diffns t\Big] - u(s, y)\\ \ge E\Big[\int_s^{s + \varepsilon}\Big\{\partial_t\varphi(t,Y^{s,y}_t) + \frac{1}{2}\tilde{\sigma}^2(Y^{s,y}_t)\partial_{xx}^2\varphi(t, Y^{s,y}_t) + g_t( F^{-1}_\nu(Y^{s,y}_t))\Big\}\diffns t \Big].
 \end{multline*}
 Since $Y^{s,y}$ is a Markov process, the left hand side above is $0$.
 Thus, multiplying both sides by $1/\varepsilon$ and taking the limit as $\varepsilon$ goes to $0$ gives
 \begin{equation*}
 	\partial_t\varphi(s,y) + \frac{1}{2}\tilde{\sigma}^2\partial^2_{xx}\varphi(s,y) + g_s( F^{-1}_\nu(y)) \ge 0.
 \end{equation*}
 That is, $u$ is a viscosity subsolution of \eqref{eq:pde}.
 A similar argument shows that $u$ is also a viscosity supersolution.
\end{proof}

\subsection{Applications to classical SDEs}

In this final section, we consider the (classical) SDE
\begin{equation}
\label{eq:class.SDE}
	X_t = x + \int_0^t\sigma(X_u)\diffns W_u + \int_0^tb(X_u)\diffns u
\end{equation}
where, $b,\sigma:\mathbb{R}\to \mathbb{R}$ are two measurable functions.
It is well known that the SDE \eqref{eq:class.SDE} is a particular case of the SDE \eqref{eq:SDE.homogene} involving the local time, when the measure $\nu$ is absolutely continuous w.r.t. the Lebesgue measure.
This observation allows us to deduce, from Section \ref{sec:uniqueness}, pathwise uniqueness, strong existence and regularity results for the classical SDE \eqref{eq:class.SDE} with measurable coefficients.

Let $N_b:=\{x: b(x) = 0\}$ and consider the following condition:
\begin{itemize}
	\item[(A2')]  $N_\sigma \subseteq N_b$ 
	and $\frac{b}{\sigma^2}1_{N^c_b} \in L^1(\mathbb{R})$. 
\end{itemize}

\begin{corollary}[Pathwise uniqueness and continuity]
\label{cor:classical.unique}
	In either of the following cases the SDE \eqref{eq:class.SDE} satisfies the pathwise uniqueness property:
	\begin{itemize}
		\item[(i)]The condition \textup{(}A2'\textup{)} is satisfied and the function $\sigma$ satisfies (LT),
		\item[(ii)] The conditions \textup{(}A2'\textup{)}, \textup{(}A3\textup{)} and \textup{(}A4\textup{)} are satisfied.
	\end{itemize}
	Moreover, if $\sigma$ is continuous and there are $A,B\ge 0$ such that $|\sigma(x)\le A(B +|x|)$, then under either of the above conditions, for all $\alpha \in [0,1/2)$ and $\varepsilon>0$ the solution $X^x$ of \eqref{eq:class.SDE} with initial condition $x$, satisfies
		\begin{equation}
		\lim_{x \to x_0}P\left(||X^x - X^{x_0}||_\alpha >\varepsilon\right) = 0 \quad\text{for all } x_0 \in \mathbb{R}.
	\end{equation}	
\end{corollary}

\begin{proof}
	Consider the measure $\nu$ given by
	\begin{equation*}
		\nu(da):=\frac{b}{\sigma^2}1_{N^c_b}(a)\diffns a.
	\end{equation*}
	Since for every process $X$ satisfying \eqref{eq:class.SDE} it holds $d\langle X\rangle_t =\sigma^2(X_t)\,dt$, by the occupation time formula, one has
	\begin{equation*}
		\int_0^tb(X_u)\diffns u = \int_0^t\frac{b}{\sigma^2}\sigma^21_{N^c_b}(X_u)\diffns u = \int_\mathbb{R}\frac{b}{\sigma^2}1_{N^c_b}(a)L^a_t(X)\diffns a = \int_\mathbb{R}L^a_t(X)\,\nu(\diffns a).
	\end{equation*}
	Thus, the SDE \eqref{eq:class.SDE} can be rewritten as
	\begin{equation}
	\label{eq:class.SDE_local}
		X_t = x + \int_0^t\sigma(X_u)\diffns W_u + \int_\mathbb{R}L^a_t(X)\,\nu(\diffns a).
	\end{equation}
	Therefore, the proof of pathwise uniqueness is a direct application of Proposition \ref{pro:lt} and Theorem \ref{thm:pathwise.unique} after identifying the SDE \eqref{eq:class.SDE} and the SDE \eqref{eq:class.SDE_local}.
	Similarly, the proof of continuity follows from Proposition \ref{pro:cont}.
\end{proof}

\begin{corollary}[Strong existence] 
	Assume that the conditions \textup{(}A2'\textup{)}, \textup{(}A3\textup{)} and \textup{(}A4\textup{)} are satisfied. 
Assume in addition that $N^c_\sigma \subseteq I_\sigma$. Then the SDE \eqref{eq:class.SDE} admits a unique strong solution.
\end{corollary}
\begin{proof}
	The result follows from Proposition \ref{pro:exist_lt} after noticing that the SDE \eqref{eq:class.SDE} corresponds to \eqref{eq:SDE.homogene} with the measure $\nu(da):= \frac{b}{\sigma^2}1_{N_b^c}(a)\diffns a$.
\end{proof}


\vspace{.71cm}

\small{  \noindent \textbf{Olivier Menoukeu Pamen}: African Institute for Mathematical Sciences, Ghana,
University of Ghana, Ghana
and University of Liverpool Institute for Financial and Actuarial Mathematics, Department of Mathematical Sciences,
L69 7ZL, United Kingdom\\
\small{\textit{E-mail address:} menoukeu@liverpool.ac.uk\\
Financial support from the Alexander von
Humboldt Foundation, under the programme financed by the German Federal Ministry of Education and Research
entitled German Research Chair No 01DG15010 is gratefully acknowledged.}
  \vspace{.2cm}

  \noindent \textbf{Youssef Ouknine}: Complex Systems Engineering and Human Systems, Mohammed VI Polytechnic University, Lot 660, Hay Moulay Rachid, Ben Guerir, 43150, Morocco and and 
Mathematics Department, FSSM, Cadi Ayyad University, , Marrakesh, 40000, Morocco \\
  \small{\textit{E-mail address:}  youssef.ouknine@um6p.ma}
  \vspace{.2cm}

\noindent Ludovic Tangpi: Princeton University, Department of Operations Research and Financial Engineering. Princeton, 08544 NJ, USA.\\
{\small\textit{E-mail address:} ludovic.tangpi@princeton.edu}.

\end{document}